\documentclass[12pt,reqno]{amsart}
\usepackage{enumerate}
\usepackage{enumitem}
\usepackage[dvipsnames, svgnames, x11names]{xcolor}
\usepackage[T1]{fontenc}
\usepackage[latin1]{inputenc}
\usepackage{amsmath}
\usepackage{cases}
\usepackage{amssymb}
\usepackage{amsthm}
\usepackage{enumerate}
\usepackage{enumitem}
\usepackage{hyperref}
\usepackage{bookmark}

\setlist[enumerate]{leftmargin=*}
\usepackage{graphics}
\usepackage{tikz}
\usetikzlibrary{patterns}
\usepackage{color}
\usepackage[most]{tcolorbox}
\tcbuselibrary{breakable, skins}

\usepackage{caption}
\usepackage{subfigure}
\usepackage[a4paper,
left=3cm, right=3cm,
top=3cm, bottom=3cm]{geometry}

\usepackage{hyperref}
\hypersetup{
	colorlinks,
	linkcolor={luh-dark-blue},
	citecolor={luh-dark-blue},
	urlcolor={orange}
}

\definecolor{luh-dark-blue}{rgb}{0.0, 0.313, 0.608}

\newtheorem{satz}{Proposition}[section]
\newtheorem{lemma}[satz]{Lemma} 
\newtheorem{remark}[satz]{Remark}
\newtheorem{theorem}[satz]{Theorem}

\newtheorem{definition}[satz]{Definition}

\newcommand{\chookrightarrow}{\mathrel{\lhook\joinrel\relbar\kern-.8ex\joinrel\lhook\joinrel\rightarrow}}

\normalsize
\numberwithin{equation}{section}

\author{Long Pei}
\address{School of Mathematics (Zhuhai), Sun Yat-sen University, 519082 Zhuhai, China}
\email{peilong@mail.sysu.edu.cn}
\author{Fengyang Xiao}
\address{School of Mathematics (Zhuhai), Sun Yat-sen University, 519082 Zhuhai, China}
\email{xiaofy5@mail2.sysu.deu.cn}
\author{Pan Zhang}
\address{School of Mathematics (Zhuhai), Sun Yat-sen University, 519082 Zhuhai, China}
\email{zhangp273@mail2.sysu.edu.cn}

\begin{document}
	
	\title[]{On the steadiness of symmetric solutions to higher order perturbations of KdV}
	
	
%
%
%

	
\begin{abstract}
We consider the traveling structure of  symmetric solutions to the Rosenau-Kawahara-RLW equation and the perturbed R-KdV-RLW equation.  Both equations are  higher order perturbations of the classical KdV equation. For the Rosenau-Kawahara-RLW equation, we prove that classical and weak solutions with a priori symmetry must be traveling solutions. For the more complicated  perturbed R-KdV-RLW equation, we classify all symmetric traveling solutions, and prove that there exists no nontrivial symmetric traveling solution of solitary type once dissipation or shoaling perturbations exist. This gives a new perspective for evaluating the suitableness of a model for water waves. In addition, this result illustrates the sharpness of the symmetry principle in [Int. Math. Res. Not. IMRN, 2009; Ehrnstrom, Holden \& Raynaud]  for solitary waves. 
\end{abstract}

	
	
	
	\maketitle
	
	\section{Introduction}
	
	Of concern in this paper is the symmetry and traveling properties of the Rosenau-Kawahara-RLW equation \cite{sym1511198}  and the perturbed R-KdV-RLW equation \cite{Pdswwrsm}. The Rosenau-Kawahara-RLW equation has the  form
	\begin{equation}\label{eq:org-R-K-RLW}
		v_t+a v_x +bv^m v_x +\kappa v_{xxx}-\alpha v_{xxt}+\beta v_{xxxxt}-\mu v_{xxxxx} = 0,
	\end{equation}
	where $a,b,\kappa,\mu$ are all real constants, $\alpha,\beta$ are both non-negative constants and $m$ is a positive integer. It is viewed, as explained below, as a generalization of the classical Korteweg-de Vries equation (KdV) 
	\begin{equation}\label{eq:KdV}
		v_t+a v_x +bv v_x +\kappa v_{xxx} = 0
	\end{equation}
	for unidirectional water waves. Benjamin et al. \cite{MR427868}  replaced $v_{xxx}$ by $v_{xxt}$  in  \eqref{eq:KdV}  to obviate the shortcomings of KdV, including in particular the dispersive blow-up of solution due to concentration of short-wave components. The new equation is called the regularized long wave equation\footnote{It is also referred to as  the Benjamin-Bona-Mahon equation (BBM).} (RLW), and it approximates the full water wave problem  to the same order as KdV. 
	
	\medskip
	
	Both RLW and KdV  exhibit the remarkable phenomenon of solitary traveling waves. These waves decay fast and are observed firstly by Russell in a river. Real experiments indicate later that these solitary waves  actually have a local core but  oscillate at far field, although the  oscillation is too small to be observed by Russell. Such wave, instead of  ceasing at far field as solitary wave, is called weakly nonlocal solitons or nanopeterons \cite{BOYD1991129}. This kind of wave can be characterized by including $v_{xxxxx}$ in the KdV \eqref{eq:KdV}, leading to the fifth-order KdV\footnote{It is also referred to as the Kawahara equation \cite{OscillatoryKT}.} (fKdV). The fKdV arises naturally for gravity-capillary waves when the Bond number is very close to and slightly less than $\frac{1}{3}$ \cite{MR969032}. 
	Later, the term $v_{xxxxt}$ is introduced into \eqref{eq:KdV} by Rosenau in \cite{Dddshoe} to overcome the break-down of a low-order  transmission line model for  L-C circuits when two neighbor inductors are far away from each other.

	\medskip
	
	
	As particular coefficients vanish,  \eqref{eq:org-R-K-RLW} reduces to several known models,  including KdV, RLW, Rosenau equation and Kawahara equation (see  \cite{sym1511198} for a review).
	Various explicit traveling wave solutions, including soliton solutions, have been found for  \eqref{eq:org-R-K-RLW} (see  \cite{MR4572912} and references therein). The physically relevant traveling solutions found in the above references are all symmetric, which naturally  arise the question and the main concern of this paper: do traveling solutions to  equation  \eqref{eq:org-R-K-RLW} have a priori symmetry?
	
	\medskip
	
	The study of \emph{a priori} symmetry inherent in steady solutions has long been a central topic of the full water wave problem for two-dimensional incompressible ideal fluids governed by the Euler equation. For solitary gravity waves,  a priori symmetry has been confirmed by the method of moving planes for supercritical waves \cite{Craig01011988,MR2407226} (nonexistence of subcritical solitary saves is confirmed for an arbitrary distribution of vorticity \cite{MR4271966}). For periodic irrotational waves  \cite{Garabedian1965,Okamoto2001} and  waves with vorticity \cite{Constantin2004a,Constantin2004,MR2362244,MR2335760,MR2407226,Matioc2013}, similar results are also confirmed,  provided that some monotone structure  is imposed on the wave profile. Furthermore, periodic traveling wave solutions with singularity in a class of weakly dispersive systems in the shallow water regime are shown to exhibit a priori symmetry, provided the reflection criterion established in \cite{MR4550679} is satisfied. These singluar traveling solutions are connected to the Whitham's conjecture \cite{MR4002168,MR4458407} for shallow water waves, while the latter is the counterpart of  Stokes waves in the Stokes conjecture for the full water problem. Intriguingly, Maehlen and Seth construct recently asymmetric periodic waves by bifurcation from trivial solutions for a class of equations of the Whitham type for capillary-gravity waves, as detailed in \cite{MR4840497,SvenssonSeth2024}.   
	
	\medskip
	
	The above results  make it intriguing to characterize precisely the connection between symmetry and steadiness of solutions. The first result in this direction is  by Ehrnstrom et al. \cite{Ehrnstroem2009}, where a general principle is put forward to confirm the  a priori steadiness of  symmetric solutions to general evolution equations. This work is further generalized in \cite{Bruell2017}, where the steadiness of symmetric solutions is thoroughly studied for general evolution equations in higher dimensions and differential systems via three principles. In particular,  an alternative argument  by Pei \cite{Pei2023}  uncovers  how the symmetry determines the velocity and profile shape. The principle mentioned above also holds in weak solution formulation for various other dispersive models  and models in high dimensions (see also \cite{Geyer2015,MR4564051,MR4732989}).

	\medskip

	In this papar, we first confirm that all symmetric classical solutions to \eqref{eq:org-R-K-RLW} are all traveling wave solutions. Although the result is consistent with the principle in \cite{Ehrnstroem2009}, we give a proof based on the argument in \cite{MR3603270} so that the propagation speed and the shape of the wave profile are determined in an explicit way by the symmetric structure. Singular traveling waves that appear in the Stokes conjecture or the Whitham's conjecture are featured with a peak or cusp at the highest crest. We then introduce weak  formulation to describe waves with such singular structures, and prove that a weak  solution on the line with spatial symmetry must be a traveling solution for the Rosenau-Kawahara-RLW equation. 
	
	\medskip 
	
	A main, novel contribution of this work is for the perturbed R-KdV-RLW equation  (see \eqref{eq:gkdv}  in Section \ref{sec:example} for precise form), which is the other of the two models that we focus on in this paper. In fact, we introduce a new way to evaluate the scope of validity of complicated approximation models for water waves from the perspective of a priori symmetric structure of traveling solutions. It has been for long time a trend to introduce higher order derivatives in  models to uncover more physical phenomenon of real waves, or to inspire new efficient numerical algorithms.  In this context,  the perturbed R-KdV-RLW equation \eqref{eq:gkdv} has recently been introduced in \cite{Pdswwrsm} as a generalization of the Rosenau-Kawahara-RLW equation \eqref{eq:org-R-K-RLW}.  This equation is highly nonlinear and consists of perturbations corresponding to dispersive, shoaling and dissipative effects, respectively. 
	It allows for a symmetric 1-soliton \cite{Pdswwrsm} when the shoaling and dissipative effects are excluded.  However, whether \eqref{eq:gkdv}  in general has a symmetric traveling  solution of solitary type is unknown previously. We give a negative answer for this question.
	Our argument is based on a new perspective on the constraint equations that symmetric traveling solutions satisfy. To be precise, we observe the fact that the symmetry axis must travel at  the same speed as the wave profile, and then manage to use this important observation to reduce a  linear partial differential equation with non-constant coefficients (as one of the constraint equations) to a linear ordinary differential equation with constant coefficients. Then, we are able to classify all symmetric steady solutions to \eqref{eq:gkdv} when dissipation and shoaling effects are involved, and prove  that such equation does not have solitary type symmetric  traveling solutions in the classical sense. The result indicates the shortcoming of this model for one dimension water waves.  The argument is also expected to work for evaluating the scope of validity of other complicated approximation models for water waves.  It is worth to mention that although we give exact, explicit form of all possible solutions, it involves complicated, tedious calculation to verify whether  or not they do solve the original perturbed R-KdV-RLW equation for given initial data. We illustrate via one particular case how to get a  criterion for the perturbed R-KdV-RLW equation to allow for a symmetric traveling solution in that particular case. 
	
	\medskip
	
	We conclude this section by outlining the framework of this paper. In Section \ref{sec:preliminary}, we  introduce the definition of   symmetry and steadiness for general functions. 
	In Section \ref{sec:Rosenau-K-RLW_smooth}, we prove that classical symmetric solutions to the generalized Rosenqu-Kawahara-RLW equation  are traveling solutions. In Section \ref{sec:R-K-RLW_weak}, we work in weak formulation, and prove that a weak solution of the generalized Rosenqu-Kawahara-RLW equation with spatial symmetry is a traveling solution. In section \ref{sec:example}, we classify all possible classical symmetric steady solutions of  the perturbed R-KdV-RLW equation, and prove that these exists no symmetric traveling solution of solitary type when dissipative perturbations are included. Finally, we give an equivalent condition for the perturbed R-KdV-RLW equation to allow for symmetric traveling solutions for a particular set of coefficients in the equation  and given initial data.

	\section{Preliminary}\label{sec:preliminary}
	
	Denote by $I$ the interval of existence for solutions to a given equation, which is usually taken as $[0,T]$ with some $T>0$. We first introduce the definition of symmetry of a function $f(t,x)$, $(t,x)\in I \times \mathbb{R}$, with respect to  the $x$ variable and the steadiness of $f(t,x)$ in the $x$-direction. 

	\begin{definition}\label{def:symmetric_function}
		We say a function $f(t,x)$ is spatially symmetric if there exists a function $\lambda(t)\in C^1(\mathbb{R})$ such that 
		\begin{equation*}
			f(t,x)=f(t,2\lambda(t)-x)
		\end{equation*}
		for every $(t,x)\in I\times \mathbb{R}$. We call $\lambda=\lambda(t)$ the axis of symmetry and the reflection of $f$ with respect to $\lambda(t)$ is denoted by
		\begin{equation*}
			f_\lambda(t,x) := f(t,2\lambda(t)-x).
		\end{equation*}
	\end{definition}
	
	
	\begin{definition}
		We say a function $f(t,x)$ with $(t,x)\in I\times\mathbb{R}$ is  steady  if there exists a function $g$ such that    
		\begin{equation}
			f(t,x)=g(x-ct) \nonumber
		\end{equation}
		for  $(t,x) \in I\times\mathbb{R}$.
	\end{definition}
	
	We introduce some notations. The space $C^{k}(U) $, where $ U \subseteq \mathbb{R}^n $ is an open set and $k \in \mathbb{N} \cup \{\infty\} $, denotes the class of functions on $ U$ with continuous spatial derivatives of all orders up to $ k$. Let $C^{\infty}_{c} (U)  $ denote the space of smooth functions $\phi : U \rightarrow \mathbb{R}$ with compact support in $U$, and the funtions in $ C^{\infty}_{c} (U)$ are the usual test functions. The space $L^{p}(U)$, $1 \leq p < \infty$, consists of all function $f: U \rightarrow \mathbb{R}$ such that
	\begin{equation*}
		\Vert f \Vert_{L^{p}} := \left(\int_{U} \vert f \vert^{p} \mathrm{d} x\right)^{\frac{1}{p}} < \infty. 
	\end{equation*}
	We denote by $\mathcal{S}$ and $\mathcal{S}^{\prime}$ the spaces of Schwartz functions and tempered distributions, respectively. Note that the Fourier transform extends naturally from $\mathcal{S}$ to $\mathcal{S}^{\prime}$. 
	Then,  the  Sobolev space $ H^s(\mathbb{R}^n) $, $ s \in \mathbb{R} $, is the space of all tempered distributions $ u \in \mathcal{S}^{\prime}(\mathbb{R}^n) $ whose Fourier transform $ \hat{u} $ is finite in the following norm
	\begin{equation*}
		\|u\|_{H^s} := \left( \, \int_{\mathbb{R}^n} \big(1 + |\xi|^2\big)^s |\hat{u}(\xi)|^2 \, \mathrm{d}\xi \right)^{\frac{1}{2}}.
	\end{equation*}
	
	We also refer traveling solutions to steady solutions frequently in the following. For convenience, we denote by $L:=1 - \alpha \partial_x^{2}+\beta \partial_x^{4}$ and rewrite the generalized Rosenau-Kawahara-RLW equation \eqref{eq:org-R-K-RLW} into the following concise form
	\begin{equation}\label{eq:R-K-RLW}
		Lv_t+F(v,\partial_{x})=0,
	\end{equation}
	where $ F(v,\partial_{x}):=a v_x+bv^mv_x+\kappa v_{xxx}-\mu v_{xxxxx}$.
	
	\section{Classical symmetric solutions of Rosenau type equations}\label{sec:Rosenau-K-RLW_smooth}
	
	\subsection{The generalized Rosenau-Kawahara-RLW equation}
	
	In this section, we prove that classical symmetric solutions of the generalized Rosenau-Kawahara-RLW equation \eqref{eq:R-K-RLW}  must be traveling wave solutions.  This result also holds for other equations in this family, such as the RLW equation, the (generalized) Rosenau-RLW and the Rosenau-KdV-RLW (more details are given in Section \ref{sect:specific models}). 
	Note that classical solutions of physical relevance to this model should be bounded. Therefore, classical solutions can be viewed as tempered distribution.  We can then use the density of of $\mathcal{S}$ in $\mathcal{S}^{\prime}$ and the Fourier transform for tempered distributions to rewrite \eqref{eq:R-K-RLW}   as 
	\begin{equation}\label{eq:R-K-RLW rewritten with L inverse}
		v_t+k*F(v,\partial_{x})=0,
	\end{equation}
	where the convolution kernel $k$ corresponds to the inverse of the operator $L$ (it is invertible due to the non-negativity of $\alpha$ and $\beta$) and  is given via the Fourier transform  by 
	\begin{equation}
		\mathcal{F}(k*f)(\xi)=(1+\alpha\xi^{2}+\beta\xi^{4})^{-1}\mathcal{F}(f)(\xi).
		\nonumber
	\end{equation}
	Note that the symbol $(1+\alpha\xi^{2}+\beta\xi^{4})^{-1}$ is smooth with respect to  $\xi \in \mathbb{R}$,  and decays at infinity. The Paley-Wiener type theorem then guarantees that $k$ decays fast at infinity,  is singular at the origin but integrable on $\mathbb{R}$. The existence and uniqueness of the solution of equation \eqref{eq:R-K-RLW} has been studied in \cite{duran2024} for initial data in Sobolev spaces, and the result is reformulated as follows. 
	\begin{lemma}[Well-posedness for Rosenau-Kawahara-RLW]\label{lemma: well-posedness for RKRLW}
		Assume $s\geq0$. Let $u_{0}\in H^{s}(\mathbb{R})$. Then, there exists $T>0$ and a unique solution $v \in X^{s}:=C(I,H^{s}(\mathbb{R}))$ for  initial value $v_{0}\in H^{s}(\mathbb{R})$, where the norm on $X^{s}$  is defined as 
		\begin{equation*}
			\Vert v \Vert_{X^{s}} = \mathop{sup}\limits_{t \in I } \Vert v \|_{H^{s}(\mathbb{R})}. \nonumber
		\end{equation*}
	\end{lemma}
	Note that it is assumed that  $\alpha^{2}<4 \beta$ in \cite{duran2024} to guarantee the invertibility of $L$. However, the non-negativity of $\alpha$ and $\beta$ in our assumption is enough to  ensure this.   Therefore, when $s$ is selected to be sufficiently large, the existence and uniqueness of the classical solutions are guaranteed.
	
	\medskip 
	
	Our theorem for classical solutions  is stated as follows.
	\begin{theorem}\label{theorem:RKR}
		Assume the generalized Rosenau-Kawahara-RLW equation \eqref{eq:R-K-RLW rewritten with L inverse} admits at most one classical solution $v(t,x), (t,x) \in I\times \mathbb{R}$, for given initial data $v_0(x) = v(0,x)$. If $v(t,x)$ is  symmetric solution, then it is a steady solution with propagation speed $\dot{\lambda}(0)$.
	\end{theorem}
	\begin{proof} 
		Assume that the generalized Rosenau-Kawahara-RLW equation \eqref{eq:R-K-RLW rewritten with L inverse} has a symmetric classical solution, which satisfies $v(t,x)=v(t,2\lambda(t)-x)$ for some function 
		$\lambda(t) \in C^1(\mathbb{R})$. Then we have
		\begin{align}
			v_{t}(t,x)&=(v_{t}+2\dot{\lambda}(t)v_{x})(t,2\lambda(t)-x),\label{R_K_RLW:partial_t}\\
			\partial^{n}_{x}v(t,x)&=(-1)^{n} (\partial^{n}_{x}v) (t,2\lambda(t)-x), \quad n \in \mathbb{Z}^{+} \label{R_K_RLW:partial_x}.
		\end{align}
		In addition, for smooth, bounded function $f:I\times\mathbb{R}\to \mathbb{R}$ that is spatially symmetric with respect to $x=\lambda(t)$, we get from \eqref{R_K_RLW:partial_x}  that 
		\begin{equation}\label{eq:convolution with derivatives and symmetry}
			[k*(\partial_{x}^{n}f)](t,x)=\int_{\mathbb{R}}k(y)\partial_{x}^{n}f(t,x-y)dy=(-1)^{n}[k*(\partial_{x}^{n}f)](t,2\lambda(t)-x) 
		\end{equation}
		for $n\in \mathbb{N}\cup\{0\}$.
		Inserting \eqref{R_K_RLW:partial_t} - \eqref{eq:convolution with derivatives and symmetry}  into \eqref{eq:R-K-RLW rewritten with L inverse}, we get
		\begin{equation}\label{z}
			[v_t+2\dot{\lambda} v_x-k*F(v,\partial_{x})]\vert_{(t,2\lambda(t)-x)}=0.
			\nonumber
		\end{equation}
		In view of the arbitrariness of $(t,x) \in I \times \mathbb{R}$, we get from the above equation that 
		\begin{equation}\label{eqlam}
			v_t+2\dot{\lambda} v_x-k*F(v,\partial_{x})=0.
		\end{equation}
		The comparison of \eqref{eq:R-K-RLW rewritten with L inverse} with \eqref{eqlam} gives
		\begin{align}
			v_{t}+ \dot{\lambda} v_x&=0, \label{RK-RLW:linear part} \\
			\dot{\lambda} v_x-k*F(v,\partial_{x})&=0.\label{RK-RLW:nonlinear part}
		\end{align}
		Note that a solution to \eqref{RK-RLW:linear part} has the form
		\begin{equation}\label{RK-RLW:g}
			v(t,x)=g(x-\lambda(t)) =:g(X)
		\end{equation}
		for some function $g$ with sufficient smoothness. For convenience, we denote by $g^{\prime}(X):=\frac{d}{dX}g(X)$. Then, we have $\partial_{x}v(t,x)=g^{\prime}(X)$ and $\partial_{t}v(t,x)=-\dot{\lambda}g^{\prime}(X)$.	 Now, for any $(t_{1},x_{1})\in I \times \mathbb{R}$, and any $t_{2} \in I$, we can choose $x_{2} \in \mathbb{R} $ such that 
		\begin{equation*}
			x_{1}-\lambda(t_{1})=x_{2}-\lambda(t_{2}) =: \tilde{X}.
		\end{equation*}
		Inserting \eqref{RK-RLW:g} into \eqref{RK-RLW:nonlinear part}, and evaluating the latter at $(t_{1},x_{1})$ and $(t_{2},x_{2})$, respectively, we get
		\begin{align*}
			& \dot{\lambda}(t_{1}) g^{\prime}(\tilde{X})-[k*F(g,\partial_{x})](\tilde{X})=0, \\
			& \dot{\lambda}(t_{2}) g^{\prime}(\tilde{X})-[k*F(g,\partial_{x})](\tilde{X})=0.
		\end{align*}
		Comparing the two equations above, we get 
		\begin{equation}
			(\dot{\lambda}(t_{1})-\dot{\lambda}(t_{2}))g^{\prime}(\tilde{X})=0.
			\nonumber
		\end{equation}
		Due to the arbitrariness of $\tilde{X}$, we have $g^{\prime}(\tilde{X})  \not\equiv  0$ for nontrivial initial data $v(0,x)=g(x-\lambda(0))$. In this way, $\dot{\lambda}(t)$ is a constant for any $t \in I$. Therefore, the solution in  \eqref{RK-RLW:g} takes the form $v(t,x)=g(x-ct)$ for some constant $c$ which can be taken as $\dot{\lambda}(0)$. Hence, the classical symmetric solution to \eqref{eq:R-K-RLW rewritten with L inverse} must be a traveling wave solution. 
	\end{proof}
	\subsection{Specific models in the Rosenau-Kawahara-RLW family}\label{sect:specific models}
	
	In this section, we give a list of concrete models that are specific examples of the generalized Rosenau-Kawahara-RLW equation \eqref{eq:org-R-K-RLW}. It is worth mentioning that there is no rigorous derivation of these model equations yet, and we recommend the review in \cite{Pdswwrsm} and references there for more information on those equations. 
	To make the picture complete, we include briefly the background and physically relevant solutions of these equations.
	\medskip
	
	When the fifth-order mixed derivative term $v_{xxxxt}$ is included, the RLW  equation and the KdV equation extend separately to the   (generalized) Rosenau-RLW equation 
	\begin{equation}\label{eq:R-RLW}
		v_{t}+a v_{x}+(v^{m+1})_{x} - \alpha v_{xxt}+\beta v_{xxxxt}=0
	\end{equation}
	and the generalized Rosenau-KdV equation 
	\begin{equation}\label{eq:R-KdV}
		v_{t}+a v_{x}+b v^{m} v_{x} + \kappa v_{xxx}+\beta v_{xxxxt}=0.
	\end{equation}
	The combination of  \eqref{eq:R-RLW} and \eqref{eq:R-KdV}
	gives rise to the Rosenau-KdV-RLW equation 
	\begin{equation*}\label{kdvrlw}
		v_t+a v_x +bv^m v_x +\kappa v_{xxx}-\alpha v_{xxt}+\beta v_{xxxxt} = 0.
	\end{equation*}
	This equation describes shallow water waves under equal-width conditions, and allows for solitary wave solutions and shock waves through the  semi-contravariant principle. By including the fifth-order spatial derivative $v_{xxxxx}$ in  the generalized Rosenau-KdV \eqref{eq:R-KdV}, Biswas \cite{biswas2011bright}  considers the generalized Rosenau-Kawahara equation
	\begin{equation*}\label{RK}
		v_t+a v_x +bv^m v_x +\kappa v_{xxx}+\beta v_{xxxxt}-\mu v_{xxxxx} = 0,
	\end{equation*}
	and confirms the existence of bright and dark soliton solutions. The case $m=1$ corresponds to the Rosenau-Kawahara equation.

	\medskip

	\section{Symmetric solutions in weak formulation for the generalized Rosenau-Kawahara-RLW}\label{sec:R-K-RLW_weak}
	
	In this section, we consider  steady structure of symmetric solutions to the generalized Rosenau-Kawahara-RLW equation \eqref{eq:R-K-RLW} in weak formulation. As mentioned in Lemma \ref{lemma: well-posedness for RKRLW}, for initial value conditions in $H^{s}$ belonging to $s \geq 0$, \cite{duran2024} has given the existence and uniqueness of solutions to the equation \eqref{eq:R-K-RLW}. When $s \in [0,6]$, the solution to the equation \eqref{eq:R-K-RLW} is a weak solution in the distribution sense, so the existence and uniqueness of the weak solution is guaranteed.     
	We first introduce the  definition of weak solutions for general initial data in $L^{2}(\mathbb{R})$ setting.

	\begin{definition}\label{weak-def:R-K-RLW}
		A function $v(t,x) \in C(I,L^{2}(\mathbb{R})) $ is called a weak solution of the generalized  Rosenau-Kawahara-RLW equation \eqref{eq:R-K-RLW} if $v$ satisfies 
		\begin{equation} \label{weak-eq: R-K-RLW}
			\iint_{I \times \mathbb{R}} v L(\phi_{t}) +\left(a v + \frac{b}{m+1} v^{m+1}\right) \phi_{x} + \kappa v \phi_{xxx}-\mu v \phi_{xxxxx} \mathrm{d}t\mathrm{d}x=0,
		\end{equation}
		for all $\phi \in C^{\infty}_{c}(I \times \mathbb{R})$.
	\end{definition}
	For convenience, we apply the bracket notation for distributions so that $\langle v,\phi \rangle=\iint_{I \times \mathbb{R}} v \phi \mathrm{d}t \mathrm{d}x$ for $v$ and $\phi$ in Definition \ref{weak-def:R-K-RLW}. To describe the steadiness of the weak solutions, we give the following lemma.  
	\begin{lemma}\label{RK-lemma}
		If $V\in  L^{2}(\mathbb{R}) $ satisfies
		\begin{equation}\label{equation lemma}
			\int_{\mathbb{R} } -c V L(  \psi_{x}) +(a V + \frac{b}{m+1} V^{m+1}+1) \psi_{x} +\kappa V \psi_{xxx}-\mu V \psi_{xxxxx} \mathrm{d}x=0,
		\end{equation}
		for all $\psi \in C^{\infty}_{c}(\mathbb{R} ) $, then $v$ given by 
		\begin{equation} \label{eq:weak solution}
			v(t,x)=V(x-c(t-t_0))
		\end{equation}
		is a weak solution of the generalized  Rosenau-Kawahara-RLW equation \eqref{eq:R-K-RLW} for any $t_0 \in I$.
	\end{lemma}
	
	\begin{proof}
		By using the Fourier transform, It is clear that the translation map $ a \mapsto V(x+a) $ is continuous from $ \mathbb{R} $ to $ L^{2}(\mathbb{R}) $. Since $t \mapsto c(t-t_{0})$ is real analytic, it thus follows that $v$ given by \eqref{eq:weak solution} belongs to $C(I , L^{2}(\mathbb{R}))$. By definition,   $v(t,x)$ satisfies
		\begin{equation*}
			\langle v,L(\varphi_{t}) \rangle+\langle a v +\frac{b}{m+1}  v^{m+1}   , \varphi_{x} \rangle + \langle \kappa v , \varphi_{xxx} \rangle-\langle \mu v, \varphi_{xxxxx} \rangle=0. 
		\end{equation*}
		It then suffices to prove that $v(t,x)$ satisfies \eqref{weak-eq: R-K-RLW}.
		For any function $\varphi \in C^{\infty}_{c}(\mathbb{R})$, we have that
		\begin{equation}\label{one equation}
			\langle v,\varphi \rangle=\langle V,\varphi_c \rangle, \quad \langle v^2,\varphi \rangle =\langle V^2,\varphi_c \rangle 
			, \quad \langle v_x^{2} ,\varphi \rangle =\langle V_x^{2},\varphi_c \rangle,
		\end{equation}
		where we denote $\varphi_c(t,x) = \varphi(t,x+c(t-t_{0}))$.
		It is clear that the following 
		\begin{equation}\label{two equation}
			(\varphi_c)_t = (\varphi_t)_c+ c(\varphi_x)_c,\quad (\varphi_c)_x = (\varphi_x)_c
		\end{equation}
		hold.
		With \eqref{one equation} and \eqref{two equation}, we obtain 
		\begin{equation*}
			\langle v,L(\varphi_{t})\rangle  = \langle V, L(\partial_t \varphi_c -c \partial_x \varphi_c)\rangle,\,  \langle \kappa v,\partial_x^3\varphi\rangle = \langle \kappa V,\partial_x^3\varphi_c \rangle,\,\langle \mu v,\partial_x^{5}\varphi\rangle = \langle \mu V,\partial_x^5\varphi_c\rangle
		\end{equation*}
		and
		\begin{equation*}
			\langle a v +\frac{b}{m+1}v^{m+1},\varphi_{x}\rangle =\langle a V+\frac{b}{m+1}V^{m+1},\partial_x \varphi_c \rangle.
		\end{equation*}
		Since $V$ is independent of time, we conclude that 
		\begin{align}
			\langle V,L(\partial_t \varphi_c) \rangle &=\int_{ \mathbb{R} } V(x) \int_{I}L(\partial_t \varphi_c) \mathrm{d}t \mathrm{d}x \nonumber\\
			&=\int_{ \mathbb{R} } V(x)L(\varphi_c(T,x)-\varphi_c(0,x)) \mathrm{d}x =0.
		\end{align}
		Collecting the above results, we find 
		\begin{align}
			&\langle v,L(\varphi_{t}) \rangle+\langle a v+\frac{b}{m+1}  v^{m+1}  , \varphi_{x} \rangle +  \langle \kappa v , \varphi_{xxx} \rangle- \langle \mu v, \varphi_{xxxxx} \rangle \nonumber \\
			=&\langle V,L(\partial_t\varphi_c-c\partial_x\varphi_c)\rangle+\langle a V+\frac{b}{m+1}V^{m+1},\partial_x\varphi_c \rangle + \langle \kappa V,\partial_x^3\varphi_c\rangle - \langle \mu V,\partial_x^5\varphi_c\rangle \nonumber \\
			=&\iint_{I  \times \mathbb{R}} [-cV(x) L(\partial_x \varphi_c)+( a V(x)+\frac{b}{m+1}V^{m+1}(x))
			\partial_x\varphi_c\nonumber \\
			&+\kappa V(x)\partial_x^3\varphi_c-\mu V(x)\partial_x^5\varphi_c ]\mathrm{d}t\mathrm{d}x\\ 
			=&0,
			\nonumber
		\end{align}
		where in the last equality we used \eqref{equation lemma} by taking $\psi(x)=\varphi_c(t,x)$ which belongs to $C^{\infty}_{c}(\mathbb{R})$ for each fixed $t \in I$. The lemma then follows directly.
	\end{proof}
	We now state the theorem for  the a priori symmetry of  weak solutions  to the generalized
	Rosenau-Kawahara-RLW equation. 
	\begin{theorem}\label{the weak theorem:R-K-RLW}
		Let $v$ be a weak solution of the generalized Rosenau-Kawahara-RLW  equation \eqref{eq:R-K-RLW} with initial data $v_{0}(x)=v(t_{0},x)$ such that the equation is locally well-posed in Lemma \ref{lemma: well-posedness for RKRLW}. If $v(t,x)$ is symmetric, then it is a steady solution with speed $\dot{\lambda}(t_{0})$.
	\end{theorem}
	\begin{proof}
		Consider test functions $\phi \in C^{\infty}_{c}(I \times \mathbb{R})$. 
		We first claim that $T_{\lambda}: \phi \mapsto \phi_{\lambda}=\phi(t,2\lambda(t)-x)$ is a bijection on $C^{1}_{c}(I,C^{\infty}_{c}(\mathbb{R}))$. In fact, since $\lambda(t) \in C^1(\mathbb{R})$, it is clear that $T_{\lambda}$ maps $C^{1}_{c}(I,C^{\infty}_{c}(\mathbb{R}))$ into itself. Moreover, for any $\phi \in C^{1}_{c}(I,C^{\infty}_{c}(\mathbb{R}))$, there exists $\Tilde{\phi}:=\phi(t,2\lambda-x)$ such that $T_{\lambda} \Tilde{\phi}=\phi$. This confirms the surjectivity of $T_{\lambda}$.
		Additionally, if $T_{\lambda}(\phi)=T_{\lambda}(\psi)$ for some $\phi,\psi \in C^{1}_{c}(I,C^{\infty}_{c}(\mathbb{R}))$, $\mathrm{ i.e., } \, \phi_{\lambda}=\psi_{\lambda}$, then $\phi=T_{\lambda}(\phi_{\lambda})=T_{\lambda}(\psi_{\lambda})=\psi $. This confirms the injectivity of $T_{\lambda}$.\\
		Let $v$ be a symmetric solution of \eqref{eq:R-K-RLW}, namely $v=v_{\lambda}$. 	By definition \ref{def:symmetric_function}, we can get 
		\begin{equation}\label{eq:weak one RKRLW}
			\langle v_{\lambda},\phi \rangle= \langle \phi_{\lambda},v \rangle,\quad \langle v^{m+1},\phi \rangle=\langle v^{m+1}_{\lambda},\phi \rangle=\langle v^{m+1},\phi_{\lambda} \rangle.
		\end{equation}
		Takeing $v$ to be $v_{\lambda}$ in \eqref{weak-eq: R-K-RLW} and inserting \eqref{eq:weak one RKRLW} into \eqref{weak-eq: R-K-RLW}, we get
		\begin{equation}\label{eq:inner RKRLW}
			\langle v,(L(\phi_{t}))_{\lambda} \rangle+\langle a v +\frac{b}{m+1}  v^{m+1}   , (\partial_{x} \phi)_{\lambda} \rangle + \langle \kappa v , (\partial_{x}^{3} \phi)_{\lambda} \rangle-\langle \mu v, (\partial_{x}^{5} \phi)_{\lambda} \rangle=0. 
		\end{equation}
		Note that for any $(t,x)\in I\times \mathbb{R}$, we have 
		\begin{equation*}
			\begin{split}
				(\phi_{\lambda})_{t}(t,x)&=[(\phi_{t})_{\lambda}-2 \dot{\lambda} (\phi_{x})_{\lambda}](t,x),\\
				(\phi_{\lambda})_{x}(t,x)&=-(\phi_{x})_{\lambda}(t,x),\\
				(\phi_{\lambda})_{y}(t,x)&=(\phi_{y})_{\lambda}(t,x),
			\end{split}
		\end{equation*}
		where $\dot{\lambda}$ denotes the time derivative  of $\lambda$.
		Therefore, we can rewrite  \eqref{eq:inner RKRLW} as 
		\begin{equation}\label{eq:rkrlw_lambda}
			\begin{split}
				0=&\langle v ,  (L(\phi_{\lambda}))_{t}-2 \dot{\lambda} (L(\phi_{\lambda}))_{x}  \rangle - \langle  a v +\frac{b}{m+1}  v^{m+1} , (\phi_{\lambda})_{x} \rangle  -  \langle \kappa v, (\phi_{\lambda})_{xxx} \rangle\\
				&+ \langle \mu u, (\phi_{\lambda})_{xxxxx}  \rangle  .
			\end{split}	
		\end{equation}
		Since the map $T_{\lambda}: \phi \rightarrow \phi_{\lambda}$ is a bijection in $ C^{1}_{c}(I,C^{\infty}_{c}(\mathbb{R}))$, we know that for any $\psi \in  C^{1}_{c}(I,C^{\infty}_{c}(\mathbb{R}))$
		there exists $\phi \in C^{1}_{c}(I,C^{\infty}_{c}(\mathbb{R}))$ such that $T_{\lambda}:\phi\mapsto \psi$, namely $\psi=\phi_{\lambda}$.
		Then, we conclude that \eqref{eq:rkrlw_lambda} actually holds with $\phi_{\lambda}$ replaced by $\psi$, and get 
		\begin{equation}\label{eq:rkrlw_psi}
			\langle v ,  \partial_{t}(L\psi)-2 \dot{\lambda} \partial_{x}(L\psi)  \rangle - \langle  a v +\frac{b}{m+1}  v^{m+1} , \psi_{x} \rangle  -  \langle \kappa v, \psi_{xxx} \rangle+ \langle \mu v, \psi_{xxxxx}  \rangle=0.
		\end{equation}
		Comparing \eqref{weak-def:R-K-RLW} with \eqref{eq:rkrlw_psi} and in view of the arbitrariness of $\phi,\psi \in C^{1}_{c}(I,C^{\infty}_{c}(\mathbb{R}))$, we can take $\psi$ to be $\phi$ in \eqref{eq:rkrlw_psi} to get 
		\begin{equation}\label{eq:RKRLW-inner}
			\langle v ,   \dot{\lambda} L(\phi_{x})  \rangle + \langle  a v +\frac{b}{m+1}  v^{m+1} , \phi_{x} \rangle  +  \langle \kappa v, \phi_{xxx} \rangle- \langle \mu v, \phi_{xxxxx}  \rangle=0.
		\end{equation}
		Now, we fix a $t_0 \in I$. For any $\psi \in C^{\infty}_{0}(\mathbb{R})$, we consider the sequence of functions $\phi_{\varepsilon}(t,x) = \psi(x) \rho_{\varepsilon}(t)$, where $\rho_{\varepsilon} \in C^{\infty}_{0}(I)$ is a mollifier with the property that $\rho_{\varepsilon} \rightarrow \delta(t-t_{0})$, the Dirac mass at $t_{0}$, as $\varepsilon \rightarrow 0$. Replacing $\phi$ by $\phi_{\varepsilon}$ in \eqref{eq:RKRLW-inner} we obtain 
		\begin{align}\label{eq:RKRLW with sequence}
			&\int_{ \mathbb{R}} (L \psi)_{x} \int_{I}  \dot{\lambda} v \rho_{\varepsilon}(t) \mathrm{d}t\mathrm{d}x +	\int_{ \mathbb{R}} \psi_{x}  \int_{I}  (a v +\frac{b}{m+1}  v^{m+1}) \rho_{\varepsilon}(t)\mathrm{d}t\mathrm{d}x \nonumber \\
			&+\int_{ \mathbb{R}}  \psi_{xxx}  \int_{I} \kappa v \rho_{\varepsilon}(t) \mathrm{d}t \mathrm{d}x-\int_{ \mathbb{R}}  \psi_{xxxxx}  \int_{I}  \mu v  \rho_{\varepsilon}(t) \mathrm{d}t \mathrm{d}x 
			=0.
		\end{align}
		Since by assumption we have that $v \in C(I,L^{2}(\mathbb{R}))$, we can get the following convergence results:
		\begin{equation*}
			\lim_{\varepsilon \to 0} \int_{I} v(t,x) \rho_{\varepsilon}(t) \mathrm{d}t = v(t_{0},x)
		\end{equation*}
		in $L^{2}(\mathbb{R})$, and 
		\begin{equation*}
			\lim_{\varepsilon \to 0} \int_{I} v^{m+1}(t,x) \rho_{\varepsilon}(t) \mathrm{d}t = v^{m+1}(t_{0},x)
		\end{equation*}
		in $L^1(\mathbb{R})$. Therefore, by letting $\varepsilon$ tend to zero, we derive from \eqref{eq:RKRLW with sequence} that $v(t_{0},x)$ satisfies Lemma \ref{RK-lemma} for $c=\dot{\lambda}(t_{0}) $ so that $\overline{v}(t,x)=v(t_{0},x-\dot{\lambda}(t_{0})(t-t_{0}))$ is a steady solution of \eqref{eq:R-K-RLW}. Since $\overline{v}(t_{0},x)=v(t_{0},x)$, the uniqueness of the solution implies that $\overline{v}(t,x)=v(t,x)$ for all time $t$. Therefore, $v$ is a steady solution with speed $\dot{\lambda}(t_{0})$.
	\end{proof}

	\section{Classification of symmetric steady solutions of the R-KDV-RLW equation}\label{sec:example}
	In this section, we consider the perturbed R-KdV-RLW equation \cite{Pdswwrsm}:
	\begin{equation}\label{eq:gkdv}
		v_t+a_{1} v_x+a_{2} v_{xxx}+a_{3} v_{xxt}+a_{4} v_{xxxxt}+a_{5}(v^n)_x=R,
	\end{equation}
	where $R$ denotes the perturbation terms given by
	\begin{align*}
		R&= b_{1} v+b_{2} v_{xx} +b_{3} v_x v_{xx}+b_{4} v^m v_x+b_{5} v v_{xxx}+b_{6} vv_x v_{xx} \\ 
		&+b_{7} v_x^3+b_{8} v_x v_{xxxx}+b_{9} v_{xx} v_{xxx}+b_{10} v_{xxxx}+b_{11} v_{xxxxx}+b_{12} vv_{xxxxx},
	\end{align*}
	where $m\in\{1,2,3,4\}$, $n >1$ is an integer,  the coefficients $a_{i} \in \mathbb{R} $,  $i=1,\ldots,5$, refer to  dispersion effect,  and the coefficients $b_{i} \in \mathbb{R}$, $i=1,\ldots,12$, refer to small  perturbations. In convention, $b_{1}$ evaluates the shoaling effect while $b_{2}, b_{10}$ evaluate the dissipative effect.
	When $b_{1}=b_{2}=b_{10}=0$, dissipation and shoaling effects are removed, and \eqref{eq:gkdv} allows for  1-soliton solution \cite{Pdswwrsm}  given by
	\begin{equation*}
		v(x,t)=A \mathrm{sech}^{\frac{4}{n-1}} [B(x-\tau t)],
	\end{equation*}
	where $A,B,\tau$ are determined by $m,n$ and  coefficients of \eqref{eq:gkdv}.
	This is a symmetric solitary traveling solution.  However, whether \eqref{eq:gkdv}  in general has a symmetric traveling  solution of solitary type is unknown previously. In this section, we give a negative answer for this question. In particular, we prove that \eqref{eq:gkdv} does not have symmetric solitary traveling solution as long as $b_{1}$, $b_{2}$ and $b_{10}$ do not vanish simultaneously. These three coefficients correspond to   perturbation terms of dissipation and shoaling types in \eqref{eq:gkdv}.  The result indicates that dissipation and shoaling effects make the symmetric structure and solitary structure  inconsistent for traveling solutions. 
	Note that a well-posedness result in Sobolev spaces is not available now due to the highly nonlinear complicated structure of this equation. So, the symmetry issue considered here is  a priori symmetry.
	\subsection{Non-existence of symmetric steady solutions of solitary type}
	We first introduce some quantities  which depend on the coefficients of the shoaling and  dissipative terms $b_{1}$, $b_{2}$ and $b_{10}$. These quantities  will be used to exclude the existence of symmetric traveling solutions of solitary type.   If $b_{1}$, $b_{2}$ and $b_{10}$ satisfy\footnote{The constraint on parameters $b_{1}$, $b_{2}$ and $b_{10}$ in \eqref{eq: the case for characterizing asymmetric solutions} is just a particular case considered in Theorem \ref{thm: no symmetric solitary solution} below.}    
	\begin{equation}\label{eq: the case for characterizing asymmetric solutions}
		b_{10} \neq 0 ,\quad b_{2}^{2}> 4b_{1}b_{10},\quad(\sqrt{b_{2}^2-4b_{1}b_{10}}-b_{2})b_{10}<0,
	\end{equation} then we define the positive quantities $\alpha_{1},\alpha_{1},\beta_{1},\beta_{2}$ by 
	\begin{equation}\label{def:alpha and beta}
		\begin{split}
			\alpha_{1}   
			&=\frac{\vert \sqrt{b_{2}^{2}-4b_{1}b_{10}}+b_{2} \vert}{2\sqrt{b_{1}b_{10}} }, \qquad \alpha_{2}
			=\frac{\vert \sqrt{b_{2}^{2}-4b_{1}b_{10}}-b_{2} \vert}{2\sqrt{b_{1}b_{10}}},\\
			\beta_{1}
			&=\sqrt{\frac{b_{2}+\sqrt{b_{2}^{2}-4b_{1}b_{10}}}{2b_{1}}}, \qquad \beta_{2}
			=\sqrt{\frac{b_{2}-\sqrt{b_{2}^{2}-4b_{1}b_{10}}}{2b_{1}}} .
		\end{split}
	\end{equation} 
	These quantities help to state the theorem in a concise form.
	In addition, we introduce quantities $c_{0},\tilde{c}_{1},\tilde{c}_{2},\tilde{c}_{3},\tilde{c}_{4}$ given by 
	\begin{align*}
		c_{0}&= \left[1- \sin(\frac{\pi}{2} \alpha_{1}) \sin ( \frac{\pi}{2}\alpha_{2} )   \right]\left[ 1+ \cos(\pi \alpha_{1}) \right]\\
		&+\left[\cos(\frac{\pi}{2} \alpha_{2}) + \cos(\frac{\pi}{2} \alpha_{1}) \sin(\frac{\pi}{2} \alpha_{2}) \right] \sin(\pi \alpha_{1}), \\
		\tilde{c}_{1}&=\left[ 1-\sin(\frac{\pi}{2} \alpha_{1}) \sin ( \frac{\pi}{2}\alpha_{2} )  \right] \left[ \cos( \pi \alpha_{1}) v_{0}(0)-v_{0}( \pi \beta_{1} )  \right] \\
		&+ \left[ \cos ( \frac{\pi}{2} \alpha_{1} ) \sin( \frac{\pi}{2}\alpha_{2}   )v_{0}(0)- \sin (  \frac{\pi}{2}\alpha_{2}) v_{0} (\frac{\pi}{2} \beta_{1}  )  + v_{0}( \frac{\pi}{2} \beta_{2} )   \right] \sin(\pi \alpha_1),\\
		\tilde{c}_{2}&= \left[ \cos ( \frac{\pi}{2} \alpha_{1} ) v_{0}(0) -v_{0}( \frac{\pi}{2}\beta_{1} ) \right] \left[ 1+ \cos( \frac{\pi}{2} \alpha_{2} ) \sin ( \pi \alpha_{1} ) \right]\\
		&-\left[ \sin ( \frac{\pi}{2}\alpha_{1} ) \cos (\frac{\pi}{2}\alpha_{2})v_{0}(0)+v_{0}( \frac{\pi}{2}\beta_{1} ) \right]\cos ( \pi \alpha_{1} )\\
		&+\left[ \cos(\frac{\pi}{2} \alpha_{1} ) + \sin(\frac{\pi}{2} \alpha_{1} )\cos( \frac{\pi}{2} \alpha_{2} )  \right]v_{0}(\pi \beta_{1}), \\
		\tilde{c}_{3}&=\left[ 1-\sin(\frac{\pi}{2} \alpha_{1}) \sin ( \frac{\pi}{2}\alpha_{2} )  \right] \left[v_{0}(0)+v_{0}(\pi \beta_{1})\right]\\
		&+\left[ \cos( \frac{\pi}{2}\alpha_{2} )v_{0}(0)+\sin (  \frac{\pi}{2}\alpha_{2})v_{0}( \frac{\pi}{2}\beta_{1} )-v_{0}( \frac{\pi}{2} \beta_{2} )  \right] \sin(\pi \alpha_{1}),\\
		\tilde{c}_{4}&=\left[\cos ( \pi \alpha_{1}) \cos( \frac{\pi}{2} \alpha_{2} )-\cos( \frac{\pi}{2} \alpha_{1} ) \sin ( \frac{\pi}{2}\alpha_{2} ) \right]v_{0}(0) \\
		&+ \left[ 1+\cos(\pi \alpha_{1}) \right] \left[\sin ( \frac{\pi}{2} \alpha_{2} ) v_{0}( \frac{\pi}{2} \beta_{1} )-v_{0}( \frac{\pi}{2} \beta_{2} )
		\right]\\
		&-\left[ \cos( \frac{\pi}{2} \alpha_{2} )+ \cos( \frac{\pi}{2} \alpha_{1} )\sin( \frac{\pi}{2} \alpha_{2} )  \right]v_{0}(\pi \beta_{1}).
	\end{align*}
	We introduce the set 
	\begin{equation}\label{def:set A}
		\mathcal{A}=\left\{(v_{0},b_{1},b_{2},b_{10}) \vert (c_{1}^{2}+c_{2}^{2})(c_{3}^{2}+c_{4}^{2}) \neq 0 \; \mbox{and} \; \bar{\theta}_{1}-\bar{\theta}_{2} \notin  \pi  \mathbb{Z}  \right\},
	\end{equation}
	where $\bar{\theta}_{1}$ and $\bar{\theta}_{2}$ are defined by
	\begin{equation}\label{def:thetabar1 thetabar2}
		\bar{\theta}_{1}:=\theta_{1}/\beta_{2},\quad  \bar{\theta}_{2}:=\theta_{2}/\beta_{1}
	\end{equation}
	via the phase parameters $\theta_{1}, \theta_{2} \in [2k\pi, (2k+1)\pi]$ determined by
	\begin{equation}\label{def:theta1 and theta2}
		\cos\theta_{1} = \frac{c_{1}}{\sqrt{c_{1}^{2} + c_{2}^{2}}}, \quad 
		\cos\theta_{2} = \frac{c_{3}}{\sqrt{c_{3}^{2} + c_{4}^{2}}}
	\end{equation}
	for some $k \in \mathbb{Z}$, 
	and 
	\begin{equation}\label{eq:values of ci}   c_{i}=e^{(i+1)\pi}\tilde{c}_{i}/c_{0}, \qquad  i=1,...,4.
	\end{equation} 
	The set  $\mathcal{A}$  helps to  characterize  the asymmetric periodic traveling solutions below when \eqref{eq: the case for characterizing asymmetric solutions} is satisfied. 
	To proceed, we introduce a criterion for the asymmetry of the sum of two trigonometric functions with  different frequency, phase and amplitude. 
	\begin{lemma}[\cite{MR4840497}, Proposition 3.7]\label{lem:asymmetric}
		Let the integers $1 \leq k_{1} < k_{2}$ be coprime and let $r_{1},r_{2},\theta_{1},\theta_{2} \in \mathbb{R}$. Then the function
		\begin{equation*}
			u(x)=r_{1} \cos(k_{1}(x+ \theta_{1}))+r_{2} \cos (k_{2}(x+ \theta_{2}))
		\end{equation*}
		is asymmetric if and only if $r_{1}r_{2} \neq 0$ and $\theta_{1}-\theta_{2} \notin \frac{\pi}{k_{1}k_{2}} \mathbb{Z}$.
	\end{lemma}
	We are ready to state the main result for the perturbed R-KdV-RLW equation.
	\begin{theorem}\label{thm: no symmetric solitary solution}
		Suppose $b_{1}^{2}+b_{2}^{2}+b_{10}^{2}\neq 0$. Assume the perturbed R-KdV-RLW equation \eqref{eq:gkdv} admits at most one classical solution $u:I \times \mathbb{R} \to \mathbb{R}$ for given initial data
		$v_0=v(0,x)$. Then, a non-trivial symmetric traveling solution must be a linear, finite combination of functions of trigonometric type. As a consequence, there exists no non-trivial symmetric traveling solution of solitary type.  Moreover,   a classical, symmetric steady solution exists only when the coefficients of \eqref{eq:gkdv} satisfy one of the following:
		\begin{enumerate}
			\item [\emph{i)}] $b_{1}=0,b_{2}b_{10}>0$,
			\item [\emph{ii)}] $b_{2}^{2}> 4b_{1}b_{10}$, $b_{1}<0$ and $b_{10}>0$,
			\item  [\emph{iii)}] $b_{2}^{2}> 4b_{1}b_{10}$ and $(\sqrt{b_{2}^2-4b_{1}b_{10}}-b_{2})b_{10}<0$,  
			\item  [\emph{iv)}] $b_{2}^{2}= 4b_{1}b_{10}$ and $b_{2}b_{10}>0$,
			\item [\emph{v)}] $b_{10}= 0, b_{2} \neq 0$ and $b_{1}b_{2} > 0$.
			
		\end{enumerate}
		
	\end{theorem}
	
	\begin{proof}
		Assume that \eqref{eq:gkdv} has a symmetric  steady solution $v(t,x)$. We get from the symmetry of $v(t,x)$ and direct calculation that 
		\begin{equation}\label{eq:gkdv_symmetry equation}
			\begin{split}
				&\left[v_t+a_{3} v_{xxt}+a_{4} v_{xxxxt}+2\dot{\lambda}(v_x+a_{3} v_{xxx}+a_{4} v_{xxxxx})\right]\Big|_{(t,2\lambda(t)-x)}\\&=\left[a_{1} v_{x}-b_{3} v_x v_{xx}-b_{4} v^m v_x+a_{5}(v^n)_x+a_{2} v_{xxx}-b_{5} v v_{xxx}\right]\Big|_{(t,2\lambda(t)-x)} \\
				&-\left[b_{6} vv_x v_{xx}+b_{7} v_x^3+b_{8} v_x v_{xxxx}+b_{9} v_{xx} v_{xxx}+b_{11} v_{xxxxx}+b_{12} vv_{xxxxx}\right]\Big|_{(t,2\lambda(t)-x)}\\
				&+\left[b_{1} v+b_{2} v_{xx}+b_{10} v_{xxxx}\right]\Big|_{(t,2\lambda(t)-x)}.
			\end{split}
		\end{equation}
		In view of the arbitrariness of $(t,2\lambda(t)-x) \in I \times \mathbb{R}$, we compare \eqref{eq:gkdv_symmetry equation} with \eqref{eq:gkdv} to get  
		\begin{equation}\label{eq:gkdv_symmetry equation_1}
			v_t+a_{3} v_{xxt}+a_{4} v_{xxxxt}+\dot{\lambda}(v_x+a_{3} v_{xxx}+a_{4} v_{xxxxx})=b_{1} v+b_{2} v_{xx}+b_{10} v_{xxxx}
		\end{equation}
		and
		\begin{equation}\label{eq:gkdv_symmetry equation_2}
			\begin{split}
				&\dot{\lambda}(v_x+a_{3} v_{xxx}+a_{4} v_{xxxxx})\\
				&=a_{1} v_x+a_{2} v_{xxx}+a_{5}(v^n)_x -b_{3} v_x v_{xx}-b_{4} v^m v_x-b_{5} v v_{xxx}\\ 
				&-b_{6} vv_x v_{xx}-b_{7} v_x^3-b_{8} v_x v_{xxxx}-b_{9} v_{xx} v_{xxx}-b_{11} v_{xxxxx}-b_{12} vv_{xxxxx}
			\end{split}
		\end{equation}
		for any $(t,x)\in I \times \mathbb{R}$.
		We introduce the traveling ansatz  $v(t,x)=g(x-\tilde{c}t)$ for some function $g$ and $\tilde{c}\in \mathbb{R}$. A key observation is that the symmetric axis $x=\lambda(t)$ travels at exactly the same speed as the solution. Therefore, we have $\dot{\lambda}(t)=\tilde{c}$ and	$v_{t}=-\dot{\lambda}v_{x}.
		$
		It is then straightforward that 
		\begin{equation}\label{eq:P2 apply to linear equation}
			v_t+a_{3} v_{xxt}+a_{4}v_{xxxxt}=-\dot{\lambda}(v_x+a_{3} v_{xxx}+a_{4}v_{xxxxx}).
		\end{equation}
		The comparison of \eqref{eq:P2 apply to linear equation} with \eqref{eq:gkdv_symmetry equation_1}
		indicates that the wave profile $v(t,x)$ satisfies
		\begin{equation}\label{eq:gkdv_second_part}
			b_{1} v+b_{2} v_{xx}+b_{10} v_{xxxx}=0.
		\end{equation}
		In the following, we classify all possible symmetric solutions of	\eqref{eq:gkdv_second_part}. For physical relevance, we exclude solutions that become unbounded as  $|y|\to \infty$.   To start with, we introduce variable substitution  $y:=x-\tilde{c}t$ and  rewrite  \eqref{eq:gkdv_second_part} as
		\begin{equation}\label{eq:gkdv_second_part with variable y}
			b_{1} g+b_{2} g_{yy}+b_{10} g_{yyyy}=0.
		\end{equation}
		It then suffices to classify all nontrivial bounded, symmetric solution $g(y)$ of \eqref{eq:gkdv_second_part with variable y}.
		Note that  \eqref{eq:gkdv_second_part with variable y} is a  linear ordinary differential equation with constant coefficients over $y$.
		The corresponding characteristic equation has the form
		\begin{equation}\label{characteristic eq}
			b_{10} \tau^4+b_{2} \tau^2 +b_{1}=0.
		\end{equation}
		With $z:=\tau^2$, we rewrite the characteristic equation \eqref{characteristic eq} as
		\begin{equation}\label{conv_char_eq}
			b_{10} z^2+b_{2} z +b_{1}=0,
		\end{equation}
		where the roots are given by
		\begin{itemize}
			\item Quadratic case ($b_{10} \neq 0$): $z^{\pm} = \frac{-b_{2} \pm \sqrt{b_{2}^{2} - 4b_{1}b_{10}}\,}{2b_{10}}$.
			\item Degenerate case ($b_{10} = 0,\ b_{2} \neq 0$): $ z = -\frac{b_{1}}{b_{2}}$.
		\end{itemize}
		We will classify all nontrivial bounded, symmetric solution of \eqref{eq:gkdv_second_part with variable y} for these two cases, respectively.
		
		We start with the  quadratic case. In this case, the roots $z^{\pm}$ depend on the discriminant $D:=b_{2}^2-4b_{1} b_{10}$.
		Since the coefficients in the equation are all real,  $z^{\pm}$ will be either both real roots  or a pair of complex conjugate roots.
		\begin{enumerate}
			\item [] \textbf{Case I}:  $D>0$, i.e., $b_{2}^{2}> 4b_{1}b_{10}$. In this case, we have two  distinct real roots $z^{\pm}$.
			\begin{enumerate}
				\item If one of $z^{\pm}$ is zero, i.e., $b_{1}=0,b_{2}b_{10} \neq 0$, then the two roots are $0$ and $-\frac{b_{2}}{b_{10}}$. {\color{black}Without loss of generality, we assume $z^{-}=0$, } $z^{+}=\frac{-b_{2}}{b_{10}}$, then \eqref{eq:gkdv_second_part with variable y} has the following general solution
				\begin{align*}
					g(y)= \left\{ \begin{array}{ccc}
						c_{1}+c_{2}y+c_{3}e^{\sqrt{z^{+}}y}+c_{4}e^{-\sqrt{z^{+}}y},& \mbox{if} & b_{2}b_{10}<0, \\
						c_{1}+c_{2}y+c_{3} \cos( \sqrt{-z^{+}}y )+c_{4} \sin (\sqrt{-z^{+}}y) ,& \mbox{if} & b_{2}b_{10}>0.
					\end{array} \right.
				\end{align*}
				Only when $c_{2}=0$, $b_{2}b_{10}>0$, and $c_{3}^{2}+c_{4}^{2} \neq 0$, we can get a nontrivial, bounded symmetric  solution. Moreover, we can use the initial condition $v_{0}(x)=v(0,x)$ to get the exact, explicit form of the solution as follows
				\begin{equation*}
					g(y)=c_{1}+c_{3} \cos( \sqrt{-z^{+}}y )+c_{4} \sin (\sqrt{-z^{+}}y),
				\end{equation*}
				where $c_{1}=[v_{0}(0)+v_{0}(\sqrt{\frac{b_{10}}{b_{2}}}\pi )]/2$, $c_{3}=[v_{0}(0)-v_{0}(\sqrt{\frac{b_{10}}{b_{2}}}\pi )]/2$ and $c_{4}=v_{0}(\sqrt{\frac{b_{10}}{b_{2}}} \frac{\pi}{2} )-[v_{0}(0)+v_{0}(\sqrt{\frac{b_{10}}{b_{2}}}\pi )]/2$.
				\item  If $z^{\pm} > 0$, i.e., $(b_{2}+\sqrt{b_{2}^2-4b_{1}b_{10}})b_{10}<0$, then the characteristic equation \eqref{characteristic eq} has four roots 
				\begin{equation*}
					\tau_{1}^{\pm}=\pm\sqrt{z^{+}} , \qquad \tau_{2}^{\pm}= \pm \sqrt{z^{-}}.
				\end{equation*}  
				Accordingly,   \eqref{eq:gkdv_second_part with variable y} has the following general solution
				\begin{equation*}
					g(y)=c_1e^{\tau_{1}^{+} y}+c_2 e^{\tau_{1}^{-} y}+c_3 e^{\tau_{2}^{+} y}+c_4 e^{\tau_{2}^{-} y},
				\end{equation*}
				where $c_{i} \in \mathbb{R}, \; i=1,\dots,4$, are given by \eqref{eq:values of ci}. This general solution will either blow up at infinity or be trivial. In this case, no physically relevant non-trivial solution is allowed.
				\item 	If $z^{+} z^{-} <0$, i.e., $b_{1}b_{10}<0$, then the characteristic equation \eqref{characteristic eq} has four roots 
				\begin{equation*}
					\tau_{1}^{\pm}=\pm\sqrt{z^{+}} , \qquad \tau_{2}^{\pm}= \pm i\sqrt{-z^{-}}.
				\end{equation*}  
				Accordingly,   \eqref{eq:gkdv_second_part with variable y} has the following general solution
				\begin{equation*}
					g(y)=c_{1}e^{\tau_{1}^{+} y}+c_2 e^{\tau_{1}^{-} y}+c_3 \cos(\sqrt{-z^{-}}  y) + c_4 \sin(\sqrt{-z^{-}}y).
				\end{equation*}
				This solution will be unbounded and blow up at infinity if $c_{1}^{2}+ c_{2}^{2} \neq 0$. 
				Therefore, a physically relevant, nontrivial solution corresponds to $c_{1}=c_{2}=0$ but $c_{3}^{2}+c_{4}^{2}\neq 0$, and we get a symmetric, periodic solution.
				Moreover, we can use the initial condition $v_{0}(x)=v(0,x)$ to get the exact, explicit form of the solution as follows
				\begin{equation}\label{eq:symmetric solution 1}
					g(y)=c_3 \cos(\sqrt{-z^{-}}  y) + c_4 \sin(\sqrt{-z^{-}}y),
				\end{equation}
				where $c_{3}=v_{0}(0)$ and $c_{4}=v_{0}\left( \frac{\pi }{2} \beta_{2} \right)$.
				\item If $z^{\pm}<0$, i.e., $(\sqrt{b_{2}^2-4b_{1}b_{10}}-b_{2})b_{10}<0$, without loss of generality, we assume $\sqrt{-z^{+}} < \sqrt{-z^{-}}$. Then the characteristic equation \eqref{characteristic eq} has four purely imaginary roots 
				\begin{equation*}
					\tau_{1}^{\pm}=\pm i \sqrt{-z^{+}}=\pm\frac{1}{\beta_{1}}i , \qquad \tau_{2}^{\pm}= \pm i\sqrt{-z^{-}}=\pm \frac{1}{\beta_{2}} i,
				\end{equation*}  
				where $\beta_{1},\beta_{2}$ are given by \eqref{def:alpha and beta}. Hence, the general solution of  \eqref{eq:gkdv_second_part with variable y} can be expressed as
				\begin{equation}\label{eq:sum of trigonometric functions of different frequency}
					\qquad \qquad g(y)=c_{1} \cos (y/\beta_{1})+c_{2} \sin (y/ \beta_{1}) +c_{3} \cos(y/ \beta_{2})+c_{4} \sin (y/ \beta_{2}).
				\end{equation}
				Moreover, a precise  expression of symmetric steady solutions can be  given for different parameters as follows. To simplify the exposition, we define
				\begin{equation}\label{def: r1 and r2}
					r_{1}:=\sqrt{c_{1}^{2}+c_{2}^{2}}, \quad  r_{2}:=\sqrt{c_{3}^{2}+c_{4}^{2}}.
					\nonumber
				\end{equation}
				\begin{enumerate}
					\item If $r_{1}=0,r_{2} \neq 0$, then \eqref{eq:sum of trigonometric functions of different frequency} gives a symmetric periodic solution to \eqref{eq:gkdv_second_part with variable y}. Applying the initial condition $v_{0}(x)=v(0,x)$, we derive the exact, explicit form of the solution as follows
					\begin{equation*}
						g(y)=c_{3} \cos(y/\beta_{2})+c_{4} \sin (y/\beta_{2}),
					\end{equation*}
					where $c_{3}=v_{0}(0),c_{4}=v_{0}(\frac{\pi}{2}\beta_{2})$.
					\item Similarly, if $r_{1}\neq 0,r_{2} = 0$, then \eqref{eq:sum of trigonometric functions of different frequency} also admits a symmetric periodic solution  to \eqref{eq:gkdv_second_part with variable y}. Applying the initial condition $v_{0}(x)=v(0,x)$, we derive the exact, explicit form of the solution as follows
					\begin{equation*}
						g(y)=c_{1} \cos (y/\beta_{1})+c_{2} \sin ( y/ \beta_{1}) ,
					\end{equation*}
					where $c_{1}=v_{0}(0),c_{2}=v_{0}(\frac{\pi}{2} \beta_{1})$.
					\item If $r_{1}r_{2} \neq 0$, then \eqref{eq:sum of trigonometric functions of different frequency} may not be symmetric  in general. In this case, it is non-trivial to give a general dependence of the symmetric structure on the coefficients $c_{1},c_{2},c_{3},c_{4},\beta_{1}$ and $\beta_{2}$.   
					However, if $1/ \beta_{1}$ and $1/\beta_{2}$ are both positive integers and coprime, then 
					Lemma \ref{lem:asymmetric} states the  solution to be asymmetric if and only if 
					$(v_{0},b_{1},b_{2},b_{10})\notin \mathcal{A}$ as defined in \eqref{def:set A}.
					In this case, the symmetric solution to \eqref{eq:gkdv_second_part with variable y} can be rewritten as
					\begin{equation}\label{eq:gkdv_non-similar frequency solution}
						g(y)=r_{1} \cos (y/ \beta_{1}-\theta_{1})+r_{2} \cos(y/ \beta_{2}-\theta_{2}),
					\end{equation} 
					where $\theta_{1},\theta_{2}$ are given by \eqref{def:theta1 and theta2}.
					We go further in the discussion to give the explicit form of the solution  by  determining the parameters in \eqref{eq:gkdv_non-similar frequency solution} via coefficients of the perturbed R-KdV-RLW equation \eqref{eq:gkdv} and initial data $v_{0}(x)=v(0,x)$. In particular, we will calculate, via the quantities in \eqref{def:alpha and beta}-\eqref{eq:values of ci},  the parameters $c_{i}$, $i=1,...,4$ and $\beta_{j},\theta_{j}$, $j=1,2$.  For conciseness, we introduce the function 
					\[
					G(x) := (\mathrm{sgn} (x)-1)^{\mathrm{sgn}(x)+1},\quad x\in \mathbb{R},
					\]
					where $\mathrm{sgn}(x)$ denotes the usual sign function.
					\begin{enumerate}
						\item If $c_{i}\neq0$, $ i=1,\ldots,4$ and $\bar{\theta}_{1}-\bar{\theta}_{2} \in \pi \mathbb{Z}$, {\color{black}for $\bar{\theta}_{1},\bar{\theta}_{2}$ defined in \eqref{def:thetabar1 thetabar2},} then   \eqref{eq:gkdv_non-similar frequency solution} is a symmetric solution with  $c_{i}=e^{(i+1)\pi}\tilde{c_{i}}/c_{0}$,
						$i=1\ldots,4$. These $c_{i}$  and $\tilde{c}_{i}$, $i=1\ldots,4$ are given in \eqref{eq:values of ci}.  
						\item If $c_{1}=0,c_{2}c_{3}c_{4} \neq 0$ and $(\frac{\pi}{2}+2k\pi)/\beta_{2}-\bar{\theta}_{2} \in \pi \mathbb{Z},\; \forall k \in \mathbb{Z}$, then \eqref{eq:gkdv_non-similar frequency solution} can be expressed as
						\begin{equation*}
							g(y)=r_{2} \cos(y/\beta_{2}-\theta_{2})+c_{2} \sin (y/\beta_{1}),
						\end{equation*}
						where
						$c_{4}= \frac{\sin ( \frac{\pi}{2} \alpha_{2} ) \left[ v_{0} ( \frac{\pi}{2} \beta_{1} )-v_{0}(0) \cos  ( \frac{\pi}{2} \alpha_{1} ) \right] -v_{0} ( \frac{\pi}{2} \beta_{2} )  }{\sin ( \frac{\pi}{2} \alpha_{1} ) \sin ( \frac{\pi}{2} \alpha_{2} )-1}$,  $c_{3}=v_{0}(0)$ and 
						$ c_{2}=\frac{v_{0}( \frac{\pi}{2} \beta_{2} ) \sin ( \frac{\pi}{2} \alpha_{1} )+ v_{0}(0) \cos ( \frac{\pi}{2} \alpha_{1} )}{\sin ( \frac{\pi}{2} \alpha_{1} ) \sin ( \frac{\pi}{2} \alpha_{2} )-1} $.
						\item If $c_{2}=0,c_{1}c_{3}c_{4} \neq 0$ and $(2k \pi+G(c_{1})\pi)/\beta_{2} -\bar{\theta}_{2} \in \pi\mathbb{Z},\; \forall k \in \mathbb{Z}$, the solution \eqref{eq:gkdv_non-similar frequency solution} can be expressed as
						\begin{equation*}               g(y)=r_{2}\cos(y/\beta_{2}-\theta_{2})+c_{1} \cos (y/\beta_{1}),
						\end{equation*}
						where
						$c_{3}= \frac{ \left[ v_{0}(0) \cos ( \frac{\pi}{2} \alpha_{2} )-v_{0}( \frac{\pi}{2} \beta_{1} )  \right] \sin ( \frac{\pi}{2} \alpha_{1} ) + v_{0} ( \frac{\pi}{2} \beta_{1} )}{\sin ( \frac{\pi}{2} \alpha_{1} )  \cos ( \frac{\pi}{2} \alpha_{2} )}$, $c_{4}=\frac{v_{0}( \frac{\pi}{2} \beta_{1} )}{ \sin ( \frac{\pi}{2} \alpha_{1} )} $ and $c_{1}= \frac{ v_{0}( \frac{\pi}{2} \beta_{2} ) \sin ( \frac{\pi}{2} \alpha_{1} ) -v_{0} ( \frac{\pi}{2} \beta_{1} )     }{ \sin ( \frac{\pi}{2} \alpha_{1} )  \cos ( \frac{\pi}{2} \alpha_{2} ) }$.
						
						\item If $c_{3}=0,c_{1}c_{2}c_{4} \neq 0$ and $\bar{\theta}_{1}-(\frac{\pi}{2}+2k\pi)/\beta_{1} \in \pi\mathbb{Z},\; \forall k \in \mathbb{Z}$, the solution \eqref{eq:gkdv_non-similar frequency solution} can be expressed as
						\begin{equation*}
							g(y)=r_{1} \cos (y/\beta_{1}-\theta_{1})+c_{4} \sin (y/\beta_{2}),
						\end{equation*}
						where
						$c_{4}=\frac{v_{0}( \frac{\pi}{2} \beta_{1} ) \sin ( \frac{\pi}{2} \alpha_{2} )+ v_{0}(0) \cos ( \frac{\pi}{2} \alpha_{2} )}{\sin ( \frac{\pi}{2} \alpha_{2} ) \sin ( \frac{\pi}{2} \alpha_{1} )-1}$, $c_{1}=v_{0}(0)$ and $c_{2}= \frac{\sin ( \frac{\pi}{2} \alpha_{1} ) \left[ v_{0} ( \frac{\pi}{2} \beta_{2} )-v_{0}(0) \cos  ( \frac{\pi}{2} \alpha_{2} ) \right] -v_{0} ( \frac{\pi}{2} \beta_{1} )  }{\sin ( \frac{\pi}{2} \alpha_{2} ) \sin ( \frac{\pi}{2} \alpha_{1} )-1}$.
						\item If $c_{4}=0,c_{1}c_{2}c_{3} \neq 0$ and $\bar{\theta}_{1}-(2k\pi+G(c_{3})\pi)/\beta_{1} \in   \pi\mathbb{Z},\; \forall k \in \mathbb{Z}$, the solution \eqref{eq:gkdv_non-similar frequency solution} can be expressed as
						\begin{equation*}
							g(y)=r_{1} \cos (y/\beta_{1}-\theta_{1})+c_{3} \cos (y/\beta_{2}),
						\end{equation*}
						where  
						$c_{1}= \frac{ \left[ v_{0}(0) \cos ( \frac{\pi}{2} \alpha_{1} )-v_{0}( \frac{\pi}{2} \beta_{2} )  \right] \sin ( \frac{\pi}{2} \alpha_{2} ) + v_{0} ( \frac{\pi}{2} \beta_{2} )}{\sin ( \frac{\pi}{2} \alpha_{2} )  \cos ( \frac{\pi}{2} \alpha_{1} )}$, $c_{2}=\frac{v_{0}( \frac{\pi}{2} \beta_{2} )}{ \sin ( \frac{\pi}{2} \alpha_{2} )}$ and 
						$c_{3}= \frac{ v_{0}( \frac{\pi}{2} \beta_{1} ) \sin ( \frac{\pi}{2} \alpha_{2} ) -v_{0} ( \frac{\pi}{2} \beta_{2} )     }{ \sin ( \frac{\pi}{2} \alpha_{2} )  \cos ( \frac{\pi}{2} \alpha_{1} ) } $.
						\item If $c_{1}c_{3} \neq 0,c_{2}=c_{4}=0$ and $(2k_{1}+G(c_{1}))/\beta_{2}-(2k_{2}+G(c_{3}))/\beta_{1} \in \mathbb{Z} ,\; \forall k_{1},k_{2} \in \mathbb{Z}$, the solution \eqref{eq:gkdv_non-similar frequency solution} can be expressed as
						\begin{equation*}
							g(y)=c_{1}\cos(y/\beta_{1})+c_{3}\cos(y/\beta_{2}),
						\end{equation*}
						where $c_{1}=\frac{v_{0}(\frac{\pi}{2} \beta_{2})}{\cos (\frac{\pi}{2} \alpha_{2})}$, $c_{3}= \frac{v_{0}(\frac{\pi}{2} \beta_{1})}{\cos (\frac{\pi}{2} \alpha_{1})} $.
						\item If $c_{1}c_{4} \neq 0 ,c_{2}=c_{3}=0$, and $(2k_{1}+G(c_{1}))/\beta_{2}-(\frac{1}{2}+2k_{2})/\beta_{1} \in \mathbb{Z} ,\; \forall k_{1},k_{2} \in \mathbb{Z}$, the solution \eqref{eq:gkdv_non-similar frequency solution} can be expressed as
						\begin{equation*}
							g(y)=c_{1} \cos(y/\beta_{1})+c_{4}\sin(y/\beta_{2}), 
						\end{equation*}
						where $c_{1}=v_{0}(0),c_{4}=\frac{v_{0}(\frac{\pi}{2}\beta_{1})}{\sin (\frac{\pi}{2}\alpha_{1})} $.
						\item If $c_{2}c_{3} \neq 0,c_{1}=c_{4}=0$, and $(\frac{1}{2}+2k_{1})/\beta_{2}-(2k+G(c_{3}))/\beta_{1} \in \mathbb{Z} ,\; \forall k_{1},k_{2} \in \mathbb{Z}$,
						the solution \eqref{eq:gkdv_non-similar frequency solution} can be expressed as
						\begin{equation*}
							g(y)=c_{2} \sin(y/\beta_{1})+c_{3}\cos(y/\beta_{2}),
						\end{equation*}
						where $c_{2}=\frac{v_{0}(\frac{\pi}{2}\beta_{2})}{\sin (\frac{\pi }{2}\alpha_{2})}$, $c_{3}=v_{0}(0)$.
						\item If  $c_{2}c_{4}\neq 0, c_{1}=c_{3} =0$, and $k_{1}/\beta_{2}-k_{2}/\beta_{1} \in \mathbb{Z} ,\; \forall k_{1},k_{2} \in \mathbb{Z}$,
						the solution \eqref{eq:gkdv_non-similar frequency solution} can be expressed as
						\begin{equation*}
							g(y)=c_{2} \sin(y/\beta_{1})+c_{4}\sin(y/\beta_{2}),
						\end{equation*}
						where  
						$c_{2}=\frac{v_{0}(\frac{\pi}{2}\beta_{2}) \sin ( \frac{\pi }{2}\alpha_{1} )-v_{0}( \frac{\pi}{2} \beta_{1})}{\sin ( \frac{\pi}{2} \alpha_{2})\sin ( \frac{\pi }{2} \alpha_{1} )-1}$ and $c_{4}=\frac{v_{0}(\frac{\pi}{2} \beta_{1} ) \sin ( \frac{\pi }{2} \alpha_{2} )-v_{0}( \frac{\pi}{2} \beta_{2} )}{\sin ( \frac{\pi}{2} \alpha_{2} )\sin ( \frac{\pi}{2} \alpha_{1} )-1}$.
					\end{enumerate}
				\end{enumerate}
			\end{enumerate}
			\item[] \textbf{Case II}: $D=0$, i.e., $b_{2}^{2}= 4b_{1}b_{10}$ and $z=\frac{-b_{2}}{2b_{10}}$ is a real root of multiplicity two.
			\begin{enumerate}
				\item If $z >0$, i.e., $b_{2}b_{10} <0$, then the characteristic equation \eqref{characteristic eq} has two real double  roots
				\begin{equation*}
					\tau^{\pm}= \pm \sqrt{z}.
				\end{equation*}
				Accordingly,   \eqref{eq:gkdv_second_part with variable y} allows the following general solution
				\begin{equation*}
					g(y)=(c_{1}+c_{2}y)e^{\tau^{+}y}+(c_{3}+c_{4}y)e^{\tau^{-}y}.
				\end{equation*}
				The solution will blow up at infinity or be trivial. In this case, no physically relevant non-trivial solution is allowed.        
				\item If $z=0$, i.e., $b_{2}=0$, then the characteristic equation \eqref{characteristic eq} has a quadruple real root
				\begin{equation*}
					\tau= 0.
				\end{equation*}
				Accordingly, \eqref{eq:gkdv_second_part with variable y} allows the following general solution
				\begin{equation*}
					g(y)=c_{1}+c_{2}y+c_{3}y^{2}+c_{4}y^{3}.
				\end{equation*}
				The solution will blow up at infinity or be trivial. In this case, no physically relevant non-trivial solution is allowed.
				\item If $z<0$, i.e., $b_{2}b_{10}>0$, then the characteristic equation \eqref{characteristic eq} has double imaginary roots
				\begin{equation*}
					\tau^{\pm}= \pm i\sqrt{-z}.
				\end{equation*}
				Accordingly,  \eqref{eq:gkdv_second_part with variable y} allows the following general solution
				\begin{equation*}
					g(y) = \left( c_1 + c_2 y \right) \cos (\sqrt{-z}y) + ( c_3 + c_4 y ) \sin (\sqrt{-z}y).
				\end{equation*}
				It is a physically relevant, non-trivial symmetric solution if and only if $c_{1}^{2}+c_{3}^{2} \neq 0,c_{2}=c_{4} = 0$,  which corresponds to the following periodic steady solution
				\begin{equation*}
					g(y)=\sqrt{c_{1}^{2}+c_{3}^{2}} \cos (\sqrt{-z} y-\theta_{3}),
				\end{equation*}    
				where $\theta_{3}\in [0,\pi]$ is determined by $\cos \theta_{3} = \frac{c_{1}}{\sqrt{c_{1}^{2}+c_{3}^{2}}}$.
				By the initial condition $v_{0}(x)=v(0,x)$, we have $c_{1}=v_{0}(0),c_{3}=v_{0}(\pi \sqrt{\frac{b_{10}}{2b_{2}}})$.
			\end{enumerate}
			
			\item [] \textbf{Case III}: $D<0$. This case corresponds to $b_{2}^2-4b_{10} b_{1} < 0$, and  $z^{\pm}$ are a pair of complex conjugate roots. 
			Then, the characteristic equation \eqref{characteristic eq} has four complex roots 
			\begin{equation*}
				\tau_{1}^{\pm}= \nu \pm  w i , \qquad \tau_{2}^{\pm}=\pm \nu +  w i,
			\end{equation*}
			where  $\nu=\sqrt{\frac{\sqrt{4b_{10} b_{1}} - b_{2}}{4 \vert b_{10}\vert}}, w=\sqrt{\frac{\sqrt{4b_{10} b_{1}} + b_{2}}{4 \vert b_{10} \vert}} $. Note that both $\nu,w\in \mathbb{R}_{+}$.
			Hence, the general solution to \eqref{eq:gkdv_second_part with variable y} has the form
			\begin{equation*}
				g(y)=e^{\nu y}[c_1 \cos (w y)+c_2 \sin ( w y)] + e^{-\nu y} [c_3 \cos (w y)+c_4 \sin (w y) ] .
			\end{equation*}
			The solution will either blow up at infinity or be trivial. Therefore, no physically relevant non-trivial solution is allowed.  
		\end{enumerate}
		We finally turn to the  degenerate case ($b_{10}= 0, b_{2} \neq 0$). In this case,  we have $ z = -\frac{b_{1}}{b_{2}}$. Then,
		the general solution to \eqref{eq:gkdv_second_part with variable y} has the form
		\begin{equation*}
			g(y)=\left\{ \begin{array}{ccc}
				c_{1} e^{\sqrt{\vert z \vert} y}+c_{2} e^{-\sqrt{\vert z \vert} y},& \mbox{if} & \frac{b_{1}}{b_{2}}<0, \\
				c_{1} \cos (\vert z \vert y) +c_{2} \sin (\vert z \vert y), &  \mbox{if} & \frac{b_{1}}{b_{2}}>0, \\
				c_{1}+c_{2}y, & \mbox{if} & b_{1}=0.
			\end{array} \right.
		\end{equation*}
		The above solution is bounded only when $b_{1}b_{2}>0$ and $b_{10}=0$.  Therefore, a symmetric, physically relevant solution is allowed only when $c_{1}=v_{0}(0),c_{2}=v_{0}(\frac{\pi b_{2}}{2b_{1}})$.
		
		In sum, if the perturbed R-KdV-RLW equation \eqref{eq:gkdv} has a nontrivial bounded, symmetric traveling wave solution, then the solution must be of a linear combination of trigonometric functions. The necessary conditions i) - v) on the coefficients of \eqref{eq:gkdv} in the statement of this theorem  for the existence of nontrivial bounded symmetric traveling solutions have already been incorporated in the proof details above.
		
	\end{proof}
	\subsection{A admissible set of coefficients and initial data to guarantee nontrivial symmetric traveling solutions}
	The proof  of Theorem \ref{thm: no symmetric solitary solution} classifies all possible symmetric steady solutions and  gives a set of necessary conditions for such solutions. In order to extract sufficient conditions from the set of necessary conditions, we need to verify that the nontrivial symmetric steady solutions obtained from \eqref{eq:gkdv_second_part} do solve the nonlinear differential equation \eqref{eq:gkdv_symmetry equation_2}. However, this procedure involves complicated, tedious calculation in general, and a sufficient condition may be a resolution of a group of complicated algebraic equations on the coefficients of \eqref{eq:gkdv} and the  initial data. Therefore, we will not analyze all cases for the necessary conditions i) - v) in Theorem \ref{thm: no symmetric solitary solution} to seek all equivalent conditions for the existence of symmetric traveling solutions. Instead, we consider the particular case \eqref{eq:symmetric solution 1} as a example, and derive a set of simple constraints on the  coefficients of \eqref{eq:gkdv} that are convenient to verify. 
	
	It suffices to verify that the general solution \eqref{eq:symmetric solution 1}  satisfies \eqref{eq:gkdv_symmetry equation_2} with the variable substitution $y=x-\tilde{c}t$.
	
	Inserting  \eqref{eq:symmetric solution 1} into  \eqref{eq:gkdv_symmetry equation_2}  and denoting by  $\omega=\sqrt{-z^{-}}$, we obtain 
	\begin{equation}\label{eq:gkdv_symmetric equation 3}
		\begin{split}
			&\omega\left[A_{1}\sin\left(\omega y\right)+A_{2}\cos\left(\omega y\right)\right]+\frac{1}{2}\omega^{3} \left[A_{3}\sin\left(2\omega y\right)+A_{4} \cos\left(2\omega y\right)\right] \\
			&+\frac{1}{4}\omega^{3} \left[A_{5} \sin\left(3 \omega y\right)+A_{6} \cos\left(3 \omega y\right)\right]+B_{m,n}=0,
		\end{split}
	\end{equation}
	
	where $c_{3}=v_{0}(0)$,  $c_{4}=v_{0}\left( \frac{\pi }{2} \beta_{2} \right)$, $\beta_{2}$ is given by \eqref{def:alpha and beta} and 
	\begin{align*}
		A_{1}&=c_{3}\left[-a_{1}+a_{2}\omega^{2}+b_{11}\omega^{4}+\dot{\lambda}(-a_{3}\omega^{2}+a_{4}\omega^{4}+1)+\frac{1}{4} \omega^{2}(3b_{7}-b_{6})(c_{3}^{2}+c_{4}^{2})\right] ,\\
		A_{2}&=c_{4}\left[a_{1}-a_{2}\omega^{2}-b_{11}\omega^{4}-\dot{\lambda}(-a_{3}\omega^{2}+a_{4}\omega^{4}+1)-\frac{1}{4} \omega^{2} (3b_{7}-b_{6})(c_{3}^{2}+c_{4}^{2})\right], \\
		A_{3}&=(c_{4}^{2}-c_{3}^{2})[b_{3}+b_{5}-(b_{8}+b_{9}+b_{12})\omega^{2}],\\
		A_{4}&=2c_{3}c_{4}[b_{3}+b_{5}-(b_{8}+b_{9}+b_{12})\omega^{2}] ,\\
		A_{5}&=c_{3}(b_{6}+b_{7})(3c_{4}^{2}-c_{3}^{2}),\\
		A_{6}&=c_{4}(b_{6}+b_{7})(3c_{3}^{2}-c_{4}^{2}),\\
		B_{m,n}&=a_{5}\omega(n-1)\left[c_3 \cos(\omega  y) + c_4 \sin(\omega y)\right]^{n-1}\left[-c_3  \sin(\omega  y) + c_4 \cos(\omega y)\right]\\
		&-b_{4}\omega\left[c_3 \cos(\omega  y) + c_4 \sin(\omega y)\right]^{m}\left[-c_3 \sin(\omega  y) + c_4 \cos(\omega y)\right].
	\end{align*}
	Note that the above equation involves trigonometric functions of different frequencies. The idea is to use the orthogonality of the trigonometric functions. In particular, one needs to collect terms with the same frequency in  \eqref{eq:gkdv_symmetric equation 3} and then make the coefficients of each single-frequency term vanish.  Note that the term $B_{m,n}$ comes from the highly nonlinear parts of \eqref{eq:gkdv}. It contributes to low frequencies for small $m$ and $n$, but will be the only high frequency part when $m,n$ are large. So, we deal with $B_{m,n}$ first. When  $m\geq3$ or $n \geq4 $, the highest frequency term in \eqref{eq:gkdv_symmetric equation 3} only appear in $B_{m,n}$. This term is denoted by $B_{m,n}^{high}$ and given by
	\begin{align*}
		B_{m,n}^{high}&=a_{5}\frac{n-1}{2^{n-1}}\omega H(n-m-1)\left[ -c_{3}^{n} \sin\left(n\omega y\right)+c_{3}^{n-1}c_{4}  \cos\left(n\omega y\right)-c_{4}^{n-1}c_{3} \delta_{1} - c_{4}^{n}\delta_{2} \right]\\
		&-\frac{b_{4}}{2^{m}}\omega H(m-n+1) \left[ -c_{3}^{m+1} \sin\left((m+1)\omega y\right)+c_{3}^{m}c_{4}  \cos\left((m+1)\omega y\right) \right. \\
		&\left.  \hspace{9.7em}-c_{4}^{m}c_{3} \delta_{3} - c_{4}^{m+1}\delta_{4} \right],
	\end{align*}
	where $  H(x)=\left\{ \begin{array}{cc}
		1  & x \geq 0 \\
		0  & x <0
	\end{array}   \right.$ and 
	\begin{equation*}
		\delta_{1}=\left\{ \begin{array}{cc}
			(-1)^{\frac{n}{2}} \cos\left(n\omega  y\right)  & n \mbox{ is even} \\
			(-1)^{\frac{n-1}{2}} \sin\left(n\omega  y\right)  & n \mbox{ is odd} 
		\end{array}   \right. ,\quad
		\delta_{2}=\left\{ \begin{array}{cc}
			\sin\left(n\omega  y\right)  & n \mbox{ is even} \\
			\cos\left(n\omega  y\right)  & n \mbox{ is odd}
		\end{array}   \right. ,
	\end{equation*}
	\begin{equation*}
		\delta_{3}=\left\{ \begin{array}{cc}
			(-1)^{\frac{m+1}{2}} \cos\left((m+1)\omega  y\right)  & m \mbox{ is odd} \\
			(-1)^{\frac{m}{2}} \sin\left((m+1)\omega  y\right)  & m \mbox{ is even} 
		\end{array}   \right. \!,\,  \delta_{4}=\left\{ \begin{array}{cc}
			\sin\left((m+1)\omega  y\right)  & m \mbox{ is odd} \\
			\cos\left((m+1)\omega  y\right)  & m \mbox{ is even}\!
		\end{array}   \right..
	\end{equation*}
	It is worth to mention that $B_{m,n}^{high}$ is enough to determine the coefficients $a_{5}$ and  $b_{4}$. Although $B_{m,n}$ do contribute lower frequency terms which may be incorporated into the lower frequency part of \eqref{eq:gkdv_symmetric equation 3}, the orthogonality of trigonometric functions indicates that $a_{5}$ and  $b_{4}$ must vanish in this case so that $B_{m,n}$ contributes actually nothing to the lower frequency parts.  
	It remains to deal with the case when $0<m<3$ and $1<n<4$. The expressions of $B_{m,n}$, $m\in \{1,2\}$, $n\in \{2,3\}$ are given as follows:
	\begin{align*}
		B_{1,2}&=\omega (a_{5}-b_{4})\left[c_{3}c_{4} \cos \left( 2 \omega y \right) + \frac{1}{2}(c_{4}^{2}-c_{3}^{2}) \sin\left( 2 \omega y \right) \right],\\
		B_{2,2}&=-\frac{1}{4} b_{4}\omega  \left[ c_{3}(3c_{4}^{2}-c_{3}^{2}) \sin \left( 3 \omega  y \right) +c_{4}(3c_{3}^{2}-c_{4}^{2}) \cos \left( 3 \omega  y \right)\right]\\
		&+a_{5}\omega \left[  c_{3}c_{4} \cos\left(2\omega y\right)+\frac{1}{2} (c_{4}^{2}-c_{3}^{2}) \sin\left( 2\omega y \right)  \right]\\
		& -\frac{1}{4} b_{4}\omega \left[-c_{3} (c_{3}^{2}+c_{4}^{2})\sin \left( \omega  y \right) + c_{4}(c_{3}^{2}+c_{4}^{2}) \cos \left( \omega  y \right)\right],\\
		B_{1,3}&=\frac{1}{4} 2a_{5}\omega  \left[ c_{3}(3c_{4}^{2}-c_{3}^{2}) \sin \left( 3 \omega  y \right) +c_{4}(3c_{3}^{2}-c_{4}^{2}) \cos \left( 3 \omega  y \right)\right]\\
		&-b_{4}\omega \left[  c_{3}c_{4} \cos\left(2\omega y\right)+\frac{1}{2} (c_{4}^{2}-c_{3}^{2}) \sin\left( 2\omega y \right)  \right]\\
		& +\frac{1}{4} 2a_{5}\omega \left[-c_{3} (c_{3}^{2}+c_{4}^{2})\sin \left( \omega  y \right) + c_{4}(c_{3}^{2}+c_{4}^{2}) \cos \left( \omega  y \right)\right],\\
		B_{2,3}&=\frac{1}{4}\omega  \left[ 2a_{5}-b_{4} \right] \left[ c_{3}(3c_{4}^{2}-c_{3}^{2}) \sin \left( 3 \omega  y \right) +c_{4}(3c_{3}^{2}-c_{4}^{2}) \cos \left( 3 \omega  y \right)\right]\\
		&+\frac{1}{4}\omega  \left[ 2a_{5}-b_{4} \right] \left[-c_{3} (c_{3}^{2}+c_{4}^{2})\sin \left( \omega  y \right) + c_{4}(c_{3}^{2}+c_{4}^{2}) \cos \left( \omega  y \right)\right],
	\end{align*}
	By the orthogonality of the trigonometric functions, it is clear that \eqref{eq:symmetric solution 1} is a solution to \eqref{eq:gkdv_symmetry equation_2} (Hence it is a nontrivial symmetric traveling solution to the perturbed R-KdV-RLW equation \eqref{eq:gkdv}) if and only if one of the following happens:
	\begin{enumerate}
		\item If $m\geq3,n \geq4 $ and $m \neq n-1$, then
		\begin{equation*}
			a_{5}= 0,\quad b_{4}= 0,\quad A_{i} = 0 ,\; i=1,\ldots,6.
		\end{equation*}    
		\item  If $n \geq 4$ and $m=n-1$, then
		\begin{equation*}
			a_{5}(n-1)=b_{4},\quad A_{i}=0,\;i=1,\ldots,6,
		\end{equation*}
		\item If $m=1,\;n=2$, then
		\begin{align*}
			A_{1}&=0,\quad A_{2}=0,\quad (a_{5}-b_{4})(c_{4}^{2}-c_{3}^{2})+\omega^{2}A_{3}=0, \\
			A_{5}&=0,\quad A_{6}=0,\qquad \;\; 2(a_{5}-b_{4})c_{3}c_{4}+\omega^{2}A_{4}=0.
		\end{align*}
		\item If $m=2,n=2$, then
		\begin{align*}
			4A_{1}+b_{4}c_{3}(c_{3}^{2}+c_{4}^{2})&=0, \qquad 4A_{2}-b_{4}c_{4}(c_{3}^{2}+c_{4}^{2})=0,\qquad \; \omega^{2}A_{3}+a_{5}(c_{4}^{2}-c_{3}^{2})=0,\\
			\omega^{2}A_{4}+2a_{5}c_{3}c_{4}&=0,\quad \omega^{2}A_{5}-b_{4}c_{3}(3c_{4}^{2}-c_{3}^{2})=0,\quad \omega^{2}A_{6}- b_{4}c_{4}(3c_{3}^{2}-c_{4}^{2})=0.
		\end{align*}
		\item If $m=1,n=3$, then
		\begin{align*}
			2A_{1}-a_{5}c_{3}(c_{3}^{2}+c_{4}^{2})&=0,\quad\;\;\;  2A_{2}+a_{5}c_{4}(c_{3}^{2}+c_{4}^{2})=0,\qquad\; \omega^{2}A_{3}-b_{4}(c_{4}^{2}-c_{3}^{2})=0,\\
			\omega^{2}A_{4}-2b_{4}c_{3}c_{4}&=0,\; \omega^{2}A_{5}+2a_{5}c_{3}(3c_{4}^{2}-c_{3}^{2})=0,\; \omega^{2}A_{6} +2a_{5}c_{4}(3c_{3}^{2}-c_{4}^{2})=0.
		\end{align*}
		\item If $m=2,n=3$, then
		\begin{align*}
			4A_{1}-c_{3}(2a_{5}-b_{4})(c_{3}^{2}+c_{4}^{2})&=0,\qquad 4A_{2}+c_{4}(2a_{5}-b_{4})(c_{3}^{2}+c_{4}^{2})=0,\quad A_{3}=0,\\
			\omega^{2}A_{5}+c_{3}(2a_{5}-b_{4})(3c_{4}^{2}-c_{3}^{2})&=0,\quad \omega^{2}A_{6}+c_{4}(2a_{5}-b_{4})(3c_{3}^{2}-c_{4}^{2})=0,\quad A_{4}=0.
		\end{align*}
	\end{enumerate}
	\begin{remark}
		Note that the general solution \eqref{eq:gkdv_non-similar frequency solution} involves different frequencies, which will make the above procedure very complicated. Therefore, we do not proceed here to give sufficient conditions for \eqref{eq:gkdv_non-similar frequency solution} to be a symmetric traveling solution of \eqref{eq:gkdv}. Nevertheless, even when different frequencies exist simultaneously, the wave speed  may still be calculated directly for large $m$ and $n$. In fact,  we have $A_{1}=A_{2} =0$ 
		when $c_{3}^{2}+c_{4}^{2} \neq 0$, $a_{4}\omega^{4}-a_{3}\omega^{2}+1 \neq 0$ and $m\geq3$, $n \geq4 $. In this case, the constant propagation speed for the traveling solution $v(t,x)=g(x-\tilde{c}t)$ is given by 
		\begin{equation*}
			\tilde{c}=\frac{\frac{1}{4} \omega^{2}(b_{6}-3b_{7})(c_{3}^{2}+c_{4}^{2})+a_{1}-a_{2}\omega^{2}-b_{11}\omega^{4}}{a_{4}\omega^{4}-a_{3}\omega^{2}+1}.
		\end{equation*}
	\end{remark}
	
%
%
%
%
	

\end{document}